\documentclass{amsart}

\usepackage[mathscr]{eucal}
\usepackage{amssymb}
\usepackage[usenames,dvipsnames]{color}
\usepackage{amsthm}
\usepackage{bbold}
\usepackage{enumerate}
\usepackage{array}
\usepackage{amsmath}

\usepackage{hyperref}

\hypersetup{colorlinks=true,citecolor=Brown,linkcolor=Brown,urlcolor=Brown}

\numberwithin{equation}{section}
\usepackage[all]{xy}

\newdir{ >}{{}*!/-10pt/\dir{>}}

\swapnumbers 

\newtheorem{Thm}[equation]{Theorem}
\newtheorem*{Thm*}{Theorem}

\newtheorem{Lem}[equation]{Lemma}
\newtheorem{Cor}[equation]{Corollary}

\theoremstyle{remark}
\newtheorem{Def}[equation]{Definition}

\newtheorem{Exa}[equation]{Example}

\newtheorem{Conv}[equation]{Convention}

\newtheorem{Rem}[equation]{Remark}
\newtheorem{Rems}[equation]{Remarks}

\newcommand{\nc}{\newcommand}
\nc{\dmo}{\DeclareMathOperator}

\dmo{\Abelem}{Abelem}
\dmo{\Add}{Add}
\dmo{\Free}{Free}
\dmo{\Id}{Id}
\dmo{\Loc}{Loc}
\dmo{\rmR}{R}
\dmo{\Spc}{Spc}
\dmo{\thick}{thick}
\dmo{\chara}{char}%
\dmo{\Coh}{Coh}
\dmo{\coker}{coker}
\dmo{\colim}{colim}
\dmo{\cone}{cone}
\dmo{\Der}{D}
\dmo{\End}{End}
\dmo{\Ext}{Ext}
\dmo{\rmH}{H}
\dmo{\Hom}{Hom}
\dmo{\id}{id}
\dmo{\Img}{Im}
\dmo{\incl}{incl}
\dmo{\Ind}{Ind}
\dmo{\Ker}{Ker}
\dmo{\LL}{L}
\dmo{\Mod}{Mod}
\dmo{\modname}{mod}%
\dmo{\Mor}{Mor}%
\dmo{\Obj}{Obj}
\dmo{\Or}{Or}
\dmo{\pr}{pr}
\dmo{\Proj}{Proj}
\dmo{\qc}{qcoh}
\dmo{\Qcoh}{Qcoh}
\dmo{\rank}{rank}
\dmo{\Res}{Res}
\dmo{\Rname}{R}
\dmo{\SH}{SH}
\dmo{\smallb}{b}
\dmo{\smallperf}{perf}
\dmo{\Spec}{Spec}
\dmo{\Stab}{Stab}
\dmo{\stab}{stab}
\dmo{\supp}{supp}
\dmo{\supph}{supph}
\dmo{\TTR}{TTR}
\dmo{\VB}{VB}

\nc{\ababs}{{\sl ab absurdo}}
\nc{\adh}[1]{\overline{#1}}
\nc{\adhpt}[1]{\adh{\{#1\}}}
\nc{\adj}{\dashv}
\nc{\adjto}{\rightleftarrows}
\nc{\aka}{{a.\,k.\,a.}\ }
\nc{\ala}{{\sl \`a la}\ }
\nc{\AddK}{\Add^{\Sigma}(\cat K)}
\nc{\Af}{\bbA_{f}}
\nc{\AMS}{\bbA\MModcat{S}}
\nc{\bbA}{\mathbb{A}}
\nc{\bbB}{\mathbb{B}}
\nc{\bbC}{\mathbb{C}}
\nc{\bbF}{\mathbb{F}}
\nc{\bbN}{\mathbb{N}}
\nc{\bbP}{\mathbb{P}}
\nc{\bbR}{\mathbb{R}}
\nc{\bbZ}{\mathbb{Z}}
\nc{\BK}{B\MModcat{K}}
\nc{\calO}{\mathcal{O}}
\nc{\cat}[1]{\mathscr{#1}}
\nc{\cO}{\mathcal{O}}
\nc{\cV}{\mathcal{V}}
\nc{\doublequot}[3]{#1\backslash #2/#3}
\nc{\Db}{\Der^{\smallb}}
\nc{\DbG}{\Db(\kk G)}
\nc{\Displ}{\displaystyle}
\nc{\Dperf}{\Der^{\smallperf}}
\nc{\Dqc}{\Der^{\qc}}
\nc{\eg}{{e.\,g.}}
\nc{\env}{{^{\operatorname{e}}}}
\nc{\eps}{\epsilon}
\nc{\FFree}{\,\text{--}\Free}%
\nc{\FFreecat}[1]{\FFree_{\cat #1}}
\nc{\FK}{\mathcal{F}(\cat K)}
\nc{\gm}{\mathfrak{m}}
\nc{\gp}{\mathfrak{p}}
\nc{\gq}{\mathfrak{q}}
\nc{\Gsets}{G\mathsf{-sets}}
\nc{\hook}{\hookrightarrow}
\nc{\HGK}{\doublequot HGK}
\nc{\Homcat}[1]{\Hom_{\cat #1}}
\nc{\ideal}[1]{\langle #1\rangle}
\nc{\ie}{{i.e.}\ }
\nc{\into}{\mathop{\rightarrowtail}}
\nc{\inv}{^{-1}}
\nc{\isoto}{\buildrel \sim\over\to}
\nc{\isotoo}{\mathop{\buildrel \sim\over\too}}
\nc{\Idcat}[1]{\Id_{\cat{#1}}}
\nc{\kk}{\Bbbk}
\nc{\Kcat}[1]{#1\MModcat{K}}
\nc{\KP}{\cat{K}_{\cat P}}
\nc{\loccit}{{\sl loc.\ cit.}}
\nc{\Lotimes}{\otimes^{\LL}}
\nc{\mmod}{\,\text{--}\modname}%
\nc{\mmodb}{\mmod^\sbull}%
\nc{\Mid}{\,\big|\,}
\nc{\MMod}{\,\text{-}\Mod}%
\nc{\MModcat}[1]{\MMod_{\cat #1}}%
\nc{\onto}{\mathop{\twoheadrightarrow}}
\nc{\op}{{^{\operatorname{op}}}}
\nc{\oto}[1]{\overset{#1}\to}
\nc{\otoo}[1]{\overset{#1}{\,\too\,}}
\nc{\ourfrac}[2]{\genfrac{}{}{0pt}{}{\scriptstyle #1}{\scriptstyle #2}}
\nc{\ouriff}{\Leftrightarrow}
\nc{\oursetminus}{\!\smallsetminus\!}
\nc{\potimes}[1]{^{\otimes #1}}
\nc{\ptimes}[1]{^{\times #1}}
\nc{\pushout}{\textrm{\rm p.o.}}
\nc{\PZG}{\cat C_{\bbZ}(\bbZ G)}
\nc{\qp}{q_{_{\scriptstyle \cat P}}\!}%
\nc{\quadtext}[1]{\quad\textrm{#1}\quad}
\nc{\restr}[1]{_{|{\scriptstyle #1}}}
\nc{\Rcat}[1]{\Rname_{\cat #1}^\sbull}
\nc{\Rf}{\RR f}
\nc{\RK}{\Rcat{K}}
\nc{\RR}{\mathsf{R}}
\nc{\sbull}{{\scriptscriptstyle\bullet}}
\nc{\smallmatrice}[1]{\left(\begin{smallmatrix} #1 \end{smallmatrix}\right)}
\nc{\suppcat}[1]{\supp(\cat #1)}
\nc{\SET}[2]{\big\{\,#1\Mid#2\,\big\}}
\nc{\SHfin}{\SH^{\text{\rm fin}}}
\nc{\SpcAK}{\Spc(A\MModcat{K})}
\nc{\SpcK}{\Spc(\cat K)}
\nc{\SpcL}{\Spc(\cat L)}
\nc{\then}{\Rightarrow}
\nc{\tideal}[1]{\ideal{#1}}
\nc{\too}{\mathop{\longrightarrow}\limits}
\nc{\TTRK}{\TTR(\cat K)}
\nc{\uA}{\underline{A}}
\nc{\unit}{\mathbb{1}}
\nc{\unitcat}[1]{\unit_{\cat #1}}
\nc{\xytriangle}[7]{\xymatrix@C=#7em{#1\ar[r]^-{\Displ #4} & #2 \ar[r]^-{\Displ #5}&#3\ar[r]^-{\Displ #6}&T #1}}

\begin{document}


\title[The \'etale Neeman-Thomason theorem]{The derived category of an \'etale extension\\ and the separable Neeman-Thomason theorem}
\author{Paul Balmer}
\date{\today}

\address{Paul Balmer, Mathematics Department, UCLA, Los Angeles, CA 90095-1555, USA}
\email{balmer@math.ucla.edu}
\urladdr{http://www.math.ucla.edu/$\sim$balmer}

\begin{abstract}
We prove that \'etale morphisms of schemes yield separable extensions of derived categories. We then generalize the Neeman-Thomason Localization Theorem to separable extensions of triangulated categories.
\end{abstract}

\subjclass[2010]{14F20, 18E30}
\keywords{Neeman-Thomason, separable extension, \'etale, derived category}

\thanks{Research supported by NSF grant~DMS-1303073.}

\maketitle

\vskip-\baselineskip\vskip-\baselineskip


\section{Introduction}
\medbreak

Our purpose is to prove two theorems of independent interest, one about derived categories in algebraic geometry and one about general triangulated categories.

If $U\hook X$ is a Zariski open subscheme, it is well-known that the derived category of~$U$ can be described out of that of~$X$ via Bousfield localization, that is, via a purely triangular construction, not resorting to models. We generalize this result to an \'etale morphism $V\to X$ by replacing Bousfield localization by a more powerful triangular construction, namely \emph{separable extension of triangulated categories} in the sense of~\cite{Balmer11}. This is Theorem~\ref{thm:Dqc}.
Summarizing the scope of separable extensions of triangulated categories, they now cover:
\begin{enumerate}
\item Bousfield localizations (see Example~\ref{exa:sep-Bous} for explanations)
\item \'etale extensions in algebraic geometry (the above-mentioned Theorem~\ref{thm:Dqc})
\item restriction to subgroups in equivariant stable homotopy categories, in equivariant $KK$-theory and in equivariant derived categories, by~\cite{BalmerDellAmbrogioSanders14pp}.
\end{enumerate}
This profusion of sources motivates the study of separable extensions of triangulated categories \textsl{per se}.

Thus stimulated, and in view of the importance of Brown representability, we prove a general result about separable extensions of \emph{compactly-generated} triangulated categories, extending the Neeman-Thomason Localization Theorem. This is Theorem~\ref{thm:main}, where we give a simple criterion for such a separable extension to remain compactly-generated and describe what happens on compact objects.

\medbreak
\section{Compact reminder}
\label{se:NT}%
\medbreak

%
\begin{Conv}
All our schemes are assumed \emph{quasi-compact and quasi-separated}, even when not repeated. This is a very light assumption, satisfied by any noetherian or any affine scheme for instance. Recall that a topological space is quasi-separated if it admits a basis of quasi-compact open subsets. For a scheme, it means that the intersection of any two affine open subsets remains quasi-compact.
\end{Conv}

For a quasi-compact and quasi-separated scheme~$X$, let us denote by $\Dqc(X)$ the derived category of complexes of~$\cO_X$-modules with quasi-coherent homology. If the reader wants to assume $X$ separated, then $\Dqc(X)$ is simply the derived category of quasi-coherent $\cO_X$-modules itself, see~\cite[Cor.5.5]{BoekstedtNeeman93}. Let us recall two important properties of the triangulated category~$\Dqc(X)$.

First of all, for a quasi-compact open subscheme~$j:U\hook X$, the derived category $\Dqc(U)$ of the subscheme is a \emph{smashing Bousfield localization} of~$\Dqc(X)$. This means that the restriction functor $j^*:\Dqc(X)\to \Dqc(U)$ admits a fully faithful right adjoint $j_*:\Dqc(U)\to \Dqc(X)$ which commutes with arbitrary coproducts. In other words, if $Z=X-U$ denotes the closed complement of~$U$, we can identify $\Dqc(U)$ as the Verdier localization $\Dqc(X)/\Dqc_{Z}(X)$ of the ambient category by its subcategory $\Dqc_{Z}(X):=\Ker(j^*)=\SET{x\in\Dqc(X)}{x\restr{U}\cong 0\textrm{ in }\Dqc(U)}$ supported on~$Z$; furthermore the endofunctor $L:=j_*j^*$ on~$\Dqc(X)$, which receives the identity $\Id_{\Dqc(X)}\otoo{\lambda} j_*j^*$, is a Bousfield localization functor (\ie $L\lambda=\lambda L: L\isoto L^2$) and $L$ is moreover smashing, \ie $L$ commutes with coproducts. Perhaps we should also remind the reader that localization of triangulated categories (\ie inverting a class of maps) is equivalent to annihilating a subcategory of objects (the cones of those maps). Hence the traditional notation with a quotient ``$\Dqc(X)/\Dqc_Z(X)$" to mean localization. These ideas are completely standard nowadays (see~\cite{Neeman01}) and a good survey can be found in~\cite{Krause10}.

In short, the derived category of an open subscheme $U\subset X$ can be described by a purely triangular construction (localization) out of the derived category of~$X$.

The second general fact we want to remind the reader about is that the category $\Dqc(X)$ is \emph{compactly-generated}, in the following sense; see details in~\cite{Neeman01}.
\begin{Def}
\label{def:comp}%
Let $\cat S$ be a triangulated category admitting all small coproducts. An object $c\in\cat S$ is called \emph{compact} if any morphism from $c$ to a coproduct $\coprod_{i\in I}x_i$ factors via a finite sub-coproduct. The category $\cat S$ is called \emph{compactly-generated} if there exists a \emph{set} of compact objects $\cat G\subset \cat S$ such that for every $x\in \cat S$, the property $\Homcat{S}(g,x)=0$ for all~$g\in\cat G$ forces $x=0$. In this case, the subcategory $\cat S^c$ of compact objects is an essentially small thick triangulated subcategory of~$\cat S$, which is exactly $\cat S^c=\thick(\cat G)$ the thick envelope of~$\cat G$. Also, the smallest localizing subcategory of~$\cat S$ which contains $\cat S^c$ is the whole~$\cat S$. (A triangulated subcategory $\cat L\subset \cat S$ is localizing if it is closed under coproducts.) Hence the name ``compactly-generated" and the formula: $\Loc(\cat S^c)=\cat S$.
\end{Def}

In the geometric example of $\Dqc(X)$, the compact objects are exactly the perfect complexes $\Dqc(X)^c=\Dperf(X)$, that is, those complexes which are locally quasi-isomorphic to bounded complexes of vector bundles. These results were first established by Neeman~\cite{Neeman96} for $X$ separated, and in the above generality by Bondal and van den Bergh~\cite[\S\,3]{BondalVandenbergh03}.

Compactly generated triangulated categories form an important class of ``big" triangulated categories, way beyond algebraic geometry. Their most remarkable properties are the Brown representability theorem and its dual~\cite[Chap.\,8]{Neeman01} and \cite[\S\,2]{Krause02}, which insure that every (co)homological functor which maps (co)products to products is (co)representable. This beautiful theory has its roots in topology but branched out to many other settings, including motivic homotopy theory or various equivariant stable homotopy theories. In algebraic geometry, Neeman applied Brown representability to Grothendieck duality, see~\cite{Neeman96}.

\goodbreak\smallbreak

Let us now recall the interaction between the two facts recorded above.

In the landmark paper~\cite{ThomasonTrobaugh90}, Thomason proved the following result, with some ectoplasmic help from Trobaugh.  Let $U\subset X$ be a quasi-compact open subscheme and $Z=X-U$ its closed complement. Then the subcategory of compact objects~$\Dperf(U)$ is the \emph{idempotent completion} $(-)^\natural$ of the Verdier localization $\Dperf(X)/\Dperf_Z(X)$ of the ambient~$\Dperf(X)$ by the thick subcategory $\Dperf_Z(X)$ supported outside~$U$. We know that $\Dperf(U)$ cannot be a localization of $\Dperf(X)$ in general, because the map on Grothendieck groups $K_0(X)\to K_0(U)$ is not always surjective. Thomason's breakthrough was to understand that this is the only obstruction. In summary, on the ``big" categories, we have the natural localization but on the compact parts we need an idempotent completion\,:
$$
\Dqc(U)=\frac{\Dqc(X)}{\Dqc_{Z}(X)}
\quadtext{whereas}
\Dperf(U)=\Bigg(\frac{\Dperf(X)}{\Dperf_Z(X)}\Bigg)^\natural.
$$
This interplay between Bousfield localization and compact-generation was then isolated by Neeman in the following abstract result (which recovers the above by plugging $\cat S=\Dqc(X)$ and $\cat R^c=\Dperf_Z(X)$)\,:

\begin{Thm}[{Neeman-Thomason Localization Theorem~\cite[Thm.\,2.1]{Neeman92b}}]
\label{thm:NT}%
Let $\cat S$ be a compactly-generated triangulated category. Let $\cat R^c\subset \cat S^c$ be a thick subcategory of compact objects and $\cat R=\Loc(\cat R^c)$ the localizing subcategory it generates. Then the smashing Bousfield localization $\cat T=\cat S/\cat R$ remains compactly-generated and its subcategory of compact objects $\cat T^c$ is canonically the idempotent completion of the corresponding Verdier localization on compacts\,: $(\cat S^c/\cat R^c)^\natural\otoo{\simeq}\cat T^c$.
\end{Thm}

We want to extend this theory from Bousfield localizations in the Zariski topology to separable extensions in the \'etale topology.

\medbreak
\section{Separable extensions and the \'etale topology}
\label{se:towards}%
\medbreak

The Neeman-Thomason theory recalled in Section~\ref{se:NT} works well for smashing Bousfield localization of compactly-generated triangulated categories. Its incarnation in algebraic geometry covers restriction to Zariski open subschemes.

However, smashing Bousfield localization has its limits. In algebraic geometry, it can essentially describe nothing else than Zariski localization. Furthermore, if we work with equivariant triangulated categories, Bousfield localization does not produce restriction to subgroups. For instance, in modular representation theory, if $\cat S=\Stab(\kk G)$ is the stable module category of a finite group~$G$ over a field~$\kk$, there is no way to obtain the stable category $\cat T=\Stab(\kk H)$ of a subgroup $H\leq G$ by means of a localization~$\cat S/\cat R$ (except of course in the trivial cases where $\cat T=\cat S$ or~$\cat T=0$). And similarly for equivariant stable homotopy categories\,: Restriction $\SH(G)\to \SH(H)$ is never a localization. A solution to this problem has been one of the first nice applications of separable extensions of triangulated categories, in the sense of~\cite{Balmer11}. Indeed, we first proved in~\cite{Balmer12bpp} that restrictions to subgroups are separable extensions in modular representation theory. Then, in the recent~\cite{BalmerDellAmbrogioSanders14pp}, we proved similar results in many other equivariant settings, including equivariant stable homotopy theory~$\SH(G)$, or Kasparov's equivariant $KK$-theory of C*-algebras, or equivariant derived categories of schemes. Here, we want to include \'etale morphisms of schemes to this list of examples.

It is high time we recall those \emph{separable extensions of triangulated categories}.

\begin{Def}
\label{def:sep}%
Let $\cat S$ be a triangulated category (in the precise sense of~\cite{Maltsiniotis06,Kuenzer07,Balmer11}, which covers all stable homotopy categories in Nature). Consider a \emph{monad} $\bbA:\cat S\to \cat S$, that is, an endofunctor equipped with multiplication $\mu:\bbA\circ\bbA\to \bbA$ and two-sided unit $\eta:\Idcat{S}\to \bbA$ satisfying the usual rules of associative and unital ring multiplication ($\mu\circ\bbA\mu=\mu\circ\mu\bbA$ and $\id=\mu\circ\bbA\eta=\mu\circ\eta\bbA$). This is a classical notion, see~\cite[Chap.\,VI]{MacLane98}. Also recall that the monad~$\bbA$ is called \emph{separable} if it satisfies the analogue of the commutative algebra definition of separability, namely if there exists a section $\sigma:\bbA\to \bbA^2$ of multiplication ($\mu\circ\sigma=\id$) which is $\bbA$-linear on both sides ($\bbA\mu\circ\sigma\bbA=\sigma\mu=\mu\bbA\circ\bbA\sigma$). When $\cat S$ is triangulated and $\bbA$ is exact, one also requires $\sigma:\bbA\to \bbA^2$ to be compatible with suspension.
\end{Def}

\begin{Exa}
\label{exa:sep-Bous}%
Suppose that the multiplication $\mu:\bbA^2\to \bbA$ is an isomorphism. Then $\bbA$ is separable, with $\sigma=\mu\inv$. This is admittedly the ``trivial" case of separability. And yet it is already an interesting case, since it is easy to see that $\mu$ is an isomorphism if and only if $\bbA\eta$ and $\eta\bbA$ are equal isomorphisms (both equal to~$\mu\inv$). This means that $\Idcat{S}\oto{\eta}\bbA$ is a Bousfield localization functor (and $\mu$ is forced to be $(\bbA\eta)\inv$). So, Bousfield localization functors are special cases of separable monads.
\end{Exa}

\begin{Rem}
\label{rem:AMS}%
The main result of~\cite{Balmer11} is that, under the assumption that $\bbA$ is exact and separable, the category of $\bbA$-modules in~$\cat S$ remains triangulated. Let us explain this statement. Thinking of a monad as a functorial version of a ring in~$\cat S$, one defines an \emph{$\bbA$-module in~$\cat S$} as an object $x\in \cat S$ equipped with a morphism $\rho:\bbA x\to x$ such that $\rho\circ\eta_x=\id_x$ and $\rho\circ\bbA\rho=\rho\circ\mu_x$. These axioms express the usual rules for modules: $1m=m$ and $a(bm)=(ab)m$ respectively. In the same vein, a \emph{morphism} $(x,\rho)\to (x',\rho')$ of $\bbA$-modules in~$\cat S$ is simply a morphism $f:x\to x'$ in~$\cat S$ such that $\rho'\circ\bbA f=f\circ\rho$. This yields the category $\AMS$ of $\bbA$-modules in~$\cat S$. The resulting Eilenberg-Moore adjunction has been around for half-a-century~\cite{EilenbergMoore65}\,:
\begin{equation}
\label{eq:FU}%
\vcenter{\xymatrix@R=1.6em{\cat S \ar@<-.25em>[d]_-{F_\bbA} \\ \AMS  \ar@<-.25em>[u]_-{U_\bbA}}}
\end{equation}
The \emph{extension-of-scalars} functor $F_\bbA$ maps $x\in\cat S$ to the \emph{free $\bbA$-module} $(\bbA x,\mu_x)$. Its right adjoint $U_\bbA$ forgets the $\bbA$-action $(x,\varrho)\mapsto x$. The theorem of~\cite{Balmer11} says that when $\bbA$ is exact and separable, $\AMS$ admits a unique triangulation which makes $F_\bbA$ and $U_\bbA$ exact. When $\eta:\Id\to \bbA$ is a Bousfield localization functor (see Example~\ref{exa:sep-Bous}), this reproves the well-known fact that the Bousfield localization $\AMS\cong\SET{x\in \cat S}{x\textrm{ is $\bbA$-local, \ie }\eta_x\textrm{ isom.}}\cong\cat S/\Ker(L)$ is triangulated. In that case, the functor $U_\bbA:\AMS\to \cat S$ is fully faithful. For a general separable extension as in~\eqref{eq:FU}, the functor $U_\bbA$ is only faithful, \ie $F_A$ is surjective up to direct summands, \ie every $\bbA$-module~$x$ is a direct summand of a free one~$x\leq F_\bbA U_\bbA(x)$.
\end{Rem}

We are now ready to state our first result, which extends~\cite[Cor.\,6.6]{Balmer11} beyond the affine and finite case. Recall that a scheme morphism $f:V\to X$ is \emph{separated} if the diagonal $\Delta_f:V\to V\times_X V$ is a closed immersion. This is a rather weak condition on a morphism which should not be confused with the separability of Definition~\ref{def:sep}. As we shall see the latter is actually more strongly related to the fact that $f:V\to X$ is \emph{unramified}.

\begin{Thm}
\label{thm:Dqc}%
Let $f:V\to X$ be a separated \'etale morphism of quasi-compact and quasi-separated schemes. Then the monad $\bbA_f:=\Rf_*\circ f^*$ on~$\Dqc(X)$ is exact and separable (Def.\,\ref{def:sep}) and there is a unique equivalence of triangulated categories
\begin{equation}
\label{eq:Dqc}%
E:\,\Dqc(V)\,\isotoo \,\Af\MModcat{\Dqc(X)}
\end{equation}
between the derived category of~$V$ and the category of $\Af$-modules in the derived category of~$X$, under which $f^*:\Dqc(X)\to \Dqc(V)$ becomes isomorphic to extension-of-scalars $F_{\Af}$ along~$\Af$ and $\Rf_*:\Dqc(V)\to \Dqc(X)$ becomes isomorphic to~$U_{\bbA_f}$; see~\eqref{eq:FU}. Also, the monad $\Af$ commutes with arbitrary coproducts.
\end{Thm}

\begin{proof}
Details about the derived categories $\Dqc(X)$ and the adjunction $f^*\adj \Rf_*$ can be found in Lipman~\cite{Lipman09}. It is a classical fact of adjunctions~\cite[Thm.\,VI.3.1]{MacLane98} that the $(f^*,\Rf_*)$-adjunction compares to the Eilenberg-Moore adjunction~\eqref{eq:FU} for the associated monad~$\bbA_f=\Rf_*\circ f^*$ on~$\Dqc(X)$, via a unique functor~$E$
$$
\xymatrix@C=1em{
& \Dqc(X) \ar@<-.2em>[ld]_-{f^*} \ar@<-.2em>[rd]_-{F_{\bbA_f}}
\\
\Dqc(V) \ar@<-.2em>[ru]_-{\Rf_*} \ar[rr]_-{E}
&& \bbA_f\MMod_{\Dqc(X)} \ar@<-.2em>[lu]_-{U_{\bbA_f}}
}
$$
which is a morphism of adjunctions, meaning that $E\circ f^*=F_{\bbA_f}$ and $U_{\bbA_f}\circ E=\Rf_*$. Explicitly, $E$ is given by the formula
\begin{equation}
\label{eq:E}%
E(y)=\big(\Rf_*(y)\,,\, \Rf_*(\eps_y)\big)
\end{equation}
for all~$y\in\Dqc(V)$ where $\eps_y:f^*\Rf_*(y)\to y$ is the counit of the $(f^*,\Rf_*)$-adjunction. Note that $\Rf_*(\eps_y)$ is a morphism $\bbA_f(\Rf_*(y))=\Rf_*f^*\Rf_*(y)\otoo{\Rf_*(\eps)}\Rf_*(y)$ which indeed describes an $\bbA_f$-action on the object $\Rf_*(y)$ of~$\Dqc(X)$.

In this situation, by~\cite[Lemma~2.10]{Balmer12bpp}, we can prove simultaneously that $\bbA_f$ is separable and that $E$ is an equivalence by showing that the counit $\eps:f^*\Rf_*\to \Id_{\Dqc(V)}$ is a split epimorphism, that is, there exists a natural transformation $\xi:\Id_{\Dqc(V)}\to f^*\Rf_*$ such that $\eps\circ\xi=\id$. Let us do that here.

Since $f$ is \'etale, it is unramified. By EGA~\cite[IV.17.4.2, p.\,65]{EGA4}, this implies that the diagonal $\Delta_f:V\to V\times_X V$ is an open immersion. On the other hand, since we assume $f$ separated, $\Delta_f$ is also closed, hence $\Delta_f:V\hook V\times_X V$ is an isomorphism onto a closed and open subscheme. In other words, there is a disconnected decomposition $V\times_X V\simeq V\sqcup Z$ for some scheme~$Z$, in such a way that $\Delta_f:V\to V\times_X V$ is the inclusion of~$V$ in~$V\sqcup Z$. Since the two projections $\pr_i:V\times_X V\to V$ satisfy $\pr_i\circ\Delta_f=\id_V$, we have the following cartesian diagram\,:
$$
\xymatrix@=5em@R=3em{
V\times_X V \cong V\sqcup Z \ar[d]|-{\kern1em \pr_1=(\id\ p_1)} \ar[r]^-{\pr_2=}_-{(\id\ p_2)}
& V \ar[d]^f
\\
V \ar[r]^-{f}
& X
}
$$
for some morphisms $p_1,p_2:Z\to V$. We want to use flat base-change on this square. Note that since $f$ is \'etale, it is also flat. Since $X$ and $V$ are quasi-compact and quasi-separated (``concentrated" in Lipman's terminology), then so is~$f$ by~\cite[Cor.\,I.6.1.10]{EGA1}. Hence we can apply base-change~\cite[Prop.\,3.9.5]{Lipman09}, which gives us a natural isomorphism $\theta:f^*\Rf_*\isotoo\RR{\pr_1}_*\,{\pr_2}^*$ defined as follows\,:
$$
\theta:\ f^*\Rf_*\otoo{\eta_2}f^*\Rf_*\,\RR{\pr_2}_*\,{\pr_2}^*\cong f^*\Rf_*\,\RR{\pr_1}_*\,{\pr_2}^*\otoo{\eps}\RR{\pr_1}_*\,{\pr_2}^*
$$
where $\eta_2$ is the unit of the $({\pr_2}^*,\RR{\pr_2}_*)$-adjunction (on which one applies $f^*\Rf_*$), where $\eps$ is our counit (applied to $\RR{\pr_1}_*\,{\pr_2}^*$) and where the middle canonical isomorphism comes from~$f\pr_1=f\pr_2$. On the other hand, we have in our case that ${\RR\pr_2}_*\,{\pr_2}^*\cong\RR(\id\ p_2)_*\,(\id\ p_2)^*\cong\Id\oplus {\RR p_2}_*\,{p_2}^*$ and similarly ${\RR\pr_1}_*\,{\pr_2}^*\cong\Id\oplus {\RR p_1}_*\,{p_2}^*$. Under this decomposition $\eta_2$ becomes~$\smallmatrice{\id\\\star}$ (where $\star$ is the unit of the $({p_2}^*,\RR {p_2}_*)$-adjunction). So, the above isomorphism~$\theta$ is the following composition
$$
f^*\Rf_*\otoo{\smallmatrice{\id\\\star}}
f^*\Rf_*\oplus f^*\Rf_*\,{\RR p_2}_*\,{p_2}^*\
\otoo{\smallmatrice{\id&0\\0&\cong}} f^*\Rf_*\oplus f^*\Rf_*\,{\RR p_1}_*\,{p_2}^*\
\otoo{\smallmatrice{\eps&0\\0&\eps}}\ \Id\oplus {\RR p_1}_*\,{p_2}^*\!.
$$
The fact that this composite is an isomorphism proves that $\eps$ (in the upper-left corner of the last map) is a split epimorphism, as wanted.

Finally, both $f^*$ and $\Rf_*$ preserve coproducts (for the latter, one can use that its left adjoint $f^*$ preserves compacts). Hence so does~$\bbA_f=\Rf_*\circ f^*$.
\end{proof}

Note that the above result implies that $\Rf_*$ is faithful, as is every~$U_{\bbA}$. However it is not \emph{fully} faithful in general, unlike what happens for $j_*:\Dqc(U)\to \Dqc(X)$ in the case of a Zariski open immersion $j:U\hook X$.

\begin{Rem}
\label{rem:monoidal}%
The adjunction $f^*:\Dqc(X)\adjto \Dqc(V):\Rf_*$ satisfies the projection formula $\Rf_*(y)\otimes x\isoto \Rf_*(y\otimes f^*(x))$ for all $y\in\Dqc(V)$ and $x\in\Dqc(X)$; see~\cite[3.9.4]{Lipman09}. Here $\otimes$ denotes the respective left derived functors of~$\otimes_{\cO_X}$ and~$\otimes_{\cO_V}$. This formula implies that the monad $\bbA_f$ is actually isomorphic to the monad $\Rf_*(\cO_V)\otimes-$ induced by the ring object $A_f:=\Rf_*(\cO_V)$ in~$\Dqc(X)$, obtained by evaluating the monad at the $\otimes$-unit. See details in~\cite[Lemma~2.7]{BalmerDellAmbrogioSanders14pp}. One can therefore rephrase Theorem~\ref{thm:Dqc} with modules over the ring object~$A_f=\bbA_f(\unit)$ instead of the monad~$\bbA_f$. This is left to the interested reader.

Echoing~\cite[Question~4.7]{Balmer12bpp}, one may wonder if the above $A_f$ are essentially the only possible commutative separable monoids (\aka ``tt-rings") in $\Dqc(X)$.
\end{Rem}

We have chosen to follow a tensor-free treatment to avoid overloading the discussion. There are two reasons for this choice. First, even in the Zariski case one often considers $j^*$ and $j_*$ instead of the ring object~$j_*\cO_U$. But more importantly, the results of this section are a (further) motivation to study separable extensions of triangulated categories, be they tensor triangulated categories or not. For instance, the generalized Neeman-Thomason theorem of the next section will hold for general monads, not only for ring objects. In particular, it does not require a tensor structure. Summarizing the examples of separable extensions of triangulated categories, we know at this stage that they include Bousfield localizations, \'etale morphisms in algebraic geometry and, thanks to~\cite{BalmerDellAmbrogioSanders14pp}, restriction to subgroups in a broad variety of equivariant stable homotopy categories.

\medbreak
\section{The Neeman-Thomason Theorem for separable extensions}
\label{se:NT-sep}%
\medbreak

The broad array of examples summarized at the end of the previous section invites us to study separable extensions of triangulated categories for themselves. In this section, we prove a generalization of the Neeman-Thomason Theorem~\ref{thm:NT} in that setting. Let us start with the only ``new" definition (of sorts) of the paper:

\begin{Def}
\label{def:smash}%
A monad $\bbA:\cat S\to \cat S$ is \emph{smashing} if it commutes with coproducts.
\end{Def}

In view of Example~\ref{exa:sep-Bous}, this notion extends the usual notion of smashing Bousfield localization functor. The localization associated to a subcategory $\cat R=\Loc(\cat R^c)$ as in the Neeman-Thomason Theorem~\ref{thm:NT} is smashing. This classical result is attributed to Ravenel, Adams and Bousfield in Neeman~\cite{Neeman92b}, although we follow~\cite{HoveyPalmieriStrickland97} who refer to Miller~\cite{Miller92}. In algebraic geometry, we have seen in Theorem~\ref{thm:Dqc} that the monad~$\bbA_f$ which describes an \'etale extension is also smashing. Definition~\ref{def:smash} is clearly the right one for our purposes:
\begin{Thm}
\label{thm:main}%
Let $\cat S$ be a compactly-generated triangulated category (Def.\,\ref{def:comp}) and let $\bbA:\cat S\to \cat S$ be a separable exact monad (Def.\,\ref{def:sep}) which is smashing in the above sense. Then the triangulated category $\cat T=\AMS$ of $\bbA$-modules in~$\cat S$ (Rem.\,\ref{rem:AMS}) is also compactly-generated and the subcategory of compact objects $\cat T^c$ coincides with the thick subcategory generated by the image of~$\cat S^c$ under extension-of-scalars~$F_\bbA:\cat S\to \cat T$ along~$\bbA$\,:
\begin{equation}
\label{eq:main}%
\cat T^c=\thick(F_\bbA(\cat S^c))\,.
\end{equation}
If moreover, $\bbA$ preserves compacts, \ie $\bbA(\cat S^c)\subseteq \cat S^c$, then we have an equality
$$
(\AMS)^c = \bbA\MMod_{\cat S^c}
$$
of subcategories of~$\cat T=\AMS$. In that case, every compact object of~$\cat T$ is a direct summand of the image $F_\bbA(c)$ of a compact object $c\in \cat S^c$.
\end{Thm}

\begin{Lem}[See~{\cite[Lem.\,8.2]{HallRydth14pp}}]
\label{lem:comp-gen}%
Let $F:\cat S\adjto \cat T:U$ be an adjunction of exact functors between triangulated categories with arbitrary coproducts and suppose that $\cat S$ is compactly generated. If the right adjoint $U$ preserves arbitrary coproducts, then $F$ preserves compact objects. If moreover $U$ is conservative (\ie $U(t)=0\then t=0$ for any object~$t\in \cat T$) then $\cat T$ is also compactly generated and $\cat T^c=\thick(F(\cat S^c))$.
\end{Lem}

\begin{proof}
The first statement is~\cite[Thm.\,5.1]{Neeman96}. So, let us assume that $U$ is conservative and let $\cat G\subset \cat S^c$ be a set of compact generators. By the first part, the set $F(\cat G)$ consists of compact objects. Suppose now that $t\in \cat T$ is such that $\Homcat{T}(F(g),t)=0$ for all $g\in \cat G$; then $\Homcat{S}(g,U(t))=\Homcat{T}(F(g),t)=0$ by adjunction. Since $\cat G$ generates~$\cat S$, we have $U(t)=0$ and therefore $t=0$ since $U$ is conservative. Hence the set $F(\cat G)$ generates~$\cat T$ and therefore $\cat T^c=\thick F(\cat G)\subseteq \thick F(\cat S^c)\subseteq \cat T^c$.
\end{proof}

\begin{proof}[Proof of Theorem~\ref{thm:main}]
First observe that the category of $\bbA$-modules $\cat T=\AMS$ admits arbitrary coproducts in the obvious way\,: $\coprod_{i\in I}(x_i,\rho_i)=\big(\coprod_{i\in I}x_i\,,\,\coprod_{i\in I}\rho_i\big)$, using that $\bbA$ commutes with coproducts for $\coprod_{i}\rho_i$ to be an $\bbA$-action on~$\coprod_i x_i$. This is where we use that $\bbA$ is smashing. In particular, we see that $F_{\bbA}$ (unsurprisingly for a left adjoint) and $U_\bbA$ both preserve coproducts; for $U_\bbA$ it follows from the above explicit formula for coproducts in~$\cat T$. We thus fall in the general assumptions of Lemma~\ref{lem:comp-gen} and conclude that $\cat T$ is compactly generated and that $\cat T^c=\thick F_\bbA(\cat S^c)$.

Now suppose moreover that $\bbA$ preserves compact objects, so that $\bbA:\cat S^c\to \cat S^c$ is a separable exact monad on~$\cat S^c$. In particular, we can form the triangulated category $\bbA\MMod_{\cat S^c}$ of $\bbA$-modules in~$\cat S^c$, which is a full subcategory of~$\AMS=\cat T$ in the obvious canonical way. This subcategory $\bbA\MMod_{\cat S^c}$ is idempotent-complete as it is the idempotent completion of the category $F_\bbA(\cat S^c)$ of free $\bbA$-modules in~$\cat S^c$ (see~\cite[Thm.\,5.17]{Balmer11} again). It follows that $\bbA\MMod_{\cat S^c}=\thick F_\bbA(\cat S^c)=\cat T^c$ by the above discussion. Hence $\cat T^c=F_\bbA(\cat S^c)^\natural$ as claimed.
\end{proof}

\begin{Rem}
\label{rem:NT}%
Although we do not use Neeman's Theorem in our proof, we do use several of his results about compactly-generated categories. Hence our work does not supersede Neeman's but truly expands it to new territories. The separable statement~\eqref{eq:main} is slightly more complicated than the Neeman-Thomason counterpart but this is the price for increased generality. Actually, we can sketch a proof of Neeman's Theorem~\ref{thm:NT} from Theorem~\ref{thm:main} as follows. If $\cat R^c\subset \cat S^c$ then $\cat R=\Loc(\cat R^c)$ is a smashing subcategory of~$\cat S$ by Miller~\cite{Miller92} and family. Hence the associated Bousfield localization $L:\cat S\to \cat S$ commutes with coproducts, \ie $L$ is a smashing separable monad. By the above Theorem~\ref{thm:main}, $\cat T=L\MModcat{T}\cong\cat S/\cat R$ is compactly generated and its compact part is $\cat T^c=\thick(F_L(\cat S^c))$, where $F_L:\cat S\to \cat T$ is the quotient functor. Now one verifies that the induced functor $\cat S^c/\cat R^c\to \cat T^c$ is fully faithful by playing around with fractions and the fact~\cite[Thm.\,4.3.3]{Neeman01} that any morphism $c\to x$ in~$\cat S$ with $c$ compact and $x\in\cat R=\Loc(\cat R^c)$ must factor via an object of~$\cat R^c$. Combining the above results, $\cat S^c/\cat R^c\to \cat T^c$ is fully faithful and has thick image. This implies that $(\cat S^c/\cat R^c)^\natural\to (\cat T^c)^\natural=\cat T^c$ is an equivalence.
\end{Rem}

\begin{Rem}
Following-up on Remark~\ref{rem:monoidal}, we can consider the situation where the compactly-generated category~$\cat S$ is a \emph{tensor} triangulated category, such that $\otimes$ commutes with coproducts in each variable. The latter is automatic if $\cat S$ admits an internal hom functor $\cat S^\op\times \cat S\to \cat S$, right adjoint to the tensor. If $A\in\cat S$ is a separable ring object in~$\cat S$, then the associated monad $A\otimes-:\cat S\to \cat S$ is therefore automatically smashing in the sense of Definition~\ref{def:smash}. Hence the category $A\MModcat{S}$ is automatically compactly generated in that case. Assuming that $A\otimes-$ preserves compacts simply means that $A$ itself is compact, at least in the common case where the $\otimes$-unit~$\unit$ is compact. In that case, we get $(A\MModcat{S})^c=A\MMod_{\cat S^c}$.
\end{Rem}

\begin{Rems}
\label{rem:f^*-comp}%
Some comments on the interaction of Theorems~\ref{thm:Dqc} and~\ref{thm:main}.
\begin{enumerate}[(1)]
\item
Let $f:V\to X$ be a separated \'etale morphism. Then, by Theorems~\ref{thm:Dqc} and~\ref{thm:main}, $\Dperf(V)$ is the thick subcategory of $\Dqc(V)$ generated by $f^*(\Dperf(X))$. In particular, there exists a perfect complex $G$ over~$X$ such that $f^*(G)$ is a perfect generator of~$\Dqc(V)$, using~\cite[Thm.\,3.1.1]{BondalVandenbergh03}.
\smallbreak
\item
\label{it:quaff}%
The above fact is known more generally for $f$ quasi-affine, at least as folklore. The reason is that for $f$ quasi-affine $\Rf_*$ is conservative and we can apply Lemma~\ref{lem:comp-gen}, without the need for $\Rf_*$ to be faithful.
\smallbreak
\item
The above does not hold for a general morphism though, not even a smooth one. Already for projective space $f:\bbP^1_\kk\to \Spec(\kk)$, for $\kk$ a field, the subcategory $\thick(f^*(\Dperf(\kk)))$ is only a proper part of $\Dperf(\bbP^1_\kk)$, equivalent to~$\Dperf(\kk)$ under~$\RR f_*$. For instance, the object $\cO_{\bbP^1}(1)$ does not belong to it.
\end{enumerate}
\end{Rems}

\begin{Cor}
If $f:V\to X$ is \emph{finite} and \'etale then $\Dperf(V)\cong \bbA_f\MMod_{\Dperf(X)}$.
\end{Cor}

\begin{proof}
In that case, $f$ is affine hence exact and $A_f=\Rf_*(\cO_V)=f_*(\cO_V)$ is a flat and locally finitely generated $\cO_X$-module, hence a vector bundle on~$X$. In particular $\bbA_f\cong A_f\otimes-$ preserves perfect complexes and we can apply the ``moreover part" of Theorem~\ref{thm:main}.
\end{proof}

\medbreak
\noindent\textbf{Acknowledgements\,:} I thank Ivo Dell'Ambrogio and Chuck Weibel for helpful discussions, as well as Amnon Neeman for pointing out Remark~\ref{rem:f^*-comp}\,\eqref{it:quaff} and for the reference to~\cite{HallRydth14pp}, thus simplifying my old proof of Theorem~\ref{thm:main}.

\bibliographystyle{alpha}%

\begin{thebibliography}{BvdB03}

\bibitem[Bal11]{Balmer11}
Paul Balmer.
\newblock Separability and triangulated categories.
\newblock {\em Adv. Math.}, 226(5):4352--4372, 2011.

\bibitem[Bal12]{Balmer12bpp}
Paul Balmer.
\newblock Stacks of group representations.
\newblock {\em J. Eur. Math. Soc. (JEMS)}, to appear, preprint 2012.
\newblock Available at \url{http://arxiv.org/abs/1302.6290}.

\bibitem[BDS14]{BalmerDellAmbrogioSanders14pp}
Paul Balmer, Ivo Dell'Ambrogio, and Beren Sanders.
\newblock Restriction to subgroups as \'etale extensions, in topology,
  {KK}-theory and geometry.
\newblock Preprint, available online, 2014.

\bibitem[BN93]{BoekstedtNeeman93}
Marcel B{\"o}kstedt and Amnon Neeman.
\newblock Homotopy limits in triangulated categories.
\newblock {\em Compositio Math.}, 86(2):209--234, 1993.

\bibitem[BvdB03]{BondalVandenbergh03}
A.~Bondal and M.~van~den Bergh.
\newblock Generators and representability of functors in commutative and
  noncommutative geometry.
\newblock {\em Mosc. Math. J.}, 3(1):1--36, 258, 2003.

\bibitem[EM65]{EilenbergMoore65}
Samuel Eilenberg and John~C. Moore.
\newblock Adjoint functors and triples.
\newblock {\em Illinois J. Math.}, 9:381--398, 1965.

\bibitem[GD71]{EGA1}
A.~Grothendieck and J.~Dieudonn\'e.
\newblock {\em \'{E}l\'ements de g\'eom\'etrie alg\'ebrique {I}.}
\newblock Springer-Verlag, Berlin, 1971.

\bibitem[Gro67]{EGA4}
Alexandre Grothendieck.
\newblock \'{E}l\'ements de g\'eom\'etrie alg\'ebrique. {IV}. \'{E}tude locale
  des sch\'emas et des morphismes de sch\'emas {IV}.
\newblock {\em Inst. Hautes \'Etudes Sci. Publ. Math.}, (20,\,24,\,28,\,32),
  1967.

\bibitem[HPS97]{HoveyPalmieriStrickland97}
Mark Hovey, John~H. Palmieri, and Neil~P. Strickland.
\newblock Axiomatic stable homotopy theory.
\newblock {\em Mem. Amer. Math. Soc.}, 128(610), 1997.

\bibitem[HR14]{HallRydth14pp}
Jack Hall and David Rydh.
\newblock Perfect complexes on algebraic stacks.
\newblock Preprint available at \url{http://arxiv.org/abs/1405.1887}, 2014.

\bibitem[Kra02]{Krause02}
Henning Krause.
\newblock A {B}rown representability theorem via coherent functors.
\newblock {\em Topology}, 41(4):853--861, 2002.

\bibitem[Kra10]{Krause10}
Henning Krause.
\newblock Localization for triangulated categories.
\newblock In {\em Triangulated categories}, volume 375 of {\em London Math.
  Soc. Lecture Note Ser.}, page 161–235. Cambridge Univ. Press, Cambridge,
  2010.

\bibitem[K{\"u}n07]{Kuenzer07}
Matthias K{\"u}nzer.
\newblock Heller triangulated categories.
\newblock {\em Homology, Homotopy Appl.}, 9(2):233--320, 2007.

\bibitem[Lip09]{Lipman09}
Joseph Lipman.
\newblock Notes on derived functors and {G}rothendieck duality.
\newblock In {\em Foundations of {G}rothendieck duality for diagrams of
  schemes}, volume 1960 of {\em Lecture Notes in Math.}, pages 1--259.
  Springer, Berlin, 2009.

\bibitem[Mal06]{Maltsiniotis06}
Georges Maltsiniotis.
\newblock Cat\'egories triangul\'ees sup\'erieures.
\newblock Preprint available online at
  \url{http://people.math.jussieu.fr/$\sim$maltsin/ps/triansup.ps}, 2006.

\bibitem[Mil92]{Miller92}
Haynes Miller.
\newblock Finite localizations.
\newblock {\em Bol. Soc. Mat. Mexicana (2)}, 37(1-2):383--389, 1992.
\newblock Papers in honor of Jos\'e Adem (Spanish).

\bibitem[ML98]{MacLane98}
Saunders Mac~Lane.
\newblock {\em Categories for the working mathematician}, volume~5 of {\em
  Graduate Texts in Mathematics}.
\newblock Springer-Verlag, New York, 1998.

\bibitem[Nee92]{Neeman92b}
Amnon Neeman.
\newblock The connection between the {$K$}-theory localization theorem of
  {T}homason, {T}robaugh and {Y}ao and the smashing subcategories of
  {B}ousfield and {R}avenel.
\newblock {\em Ann. Sci. \'Ecole Norm. Sup. (4)}, 25(5):547--566, 1992.

\bibitem[Nee96]{Neeman96}
Amnon Neeman.
\newblock The {G}rothendieck duality theorem via {B}ousfield's techniques and
  {B}rown representability.
\newblock {\em J. Amer. Math. Soc.}, 9(1):205--236, 1996.

\bibitem[Nee01]{Neeman01}
Amnon Neeman.
\newblock {\em Triangulated categories}, volume 148 of {\em Annals of
  Mathematics Studies}.
\newblock Princeton University Press, 2001.

\bibitem[TT90]{ThomasonTrobaugh90}
R.~W. Thomason and T.~Trobaugh.
\newblock Higher algebraic {$K$}-theory of schemes and of derived categories.
\newblock In {\em The Grothendieck Festschrift, Vol.\ III}, volume~88 of {\em
  Progr. Math.}, pages 247--435. Birkh\"auser, Boston, MA, 1990.

\end{thebibliography}

\end{document}